\documentclass[11pt]{amsart}
\usepackage{amsfonts}

\usepackage[all,cmtip]{xy}
\usepackage{amsmath}
\usepackage{amsthm}
\usepackage{amssymb}
\usepackage{enumitem}	
\newtheorem{thm}{Theorem}[section]
\newtheorem{conj}{Conjecture}[section]
\newtheorem{prop}[thm]{Proposition}
\newtheorem{definition}         {Definition}
\newtheorem{lemma}[thm]{Lemma}

\newtheorem{lem}[thm]{Lemma}
\newtheorem{cor}[thm]{Corollary}

\newtheorem{remark}{Remark}
\theoremstyle{remark}

\newcommand*{\Q}{\mathbb{Q}}

\newcommand*{\Z}{\mathbb{Z}}

\newcommand*{\F}{\mathbb{F}}

\newcommand*{\Gal}{\textrm{Gal}}

\newcommand*{\ra}{\rightarrow}
\newcommand*{\Aut}{\textrm{Aut}}
\newcommand*{\ol}{\overline}
\newcommand*{\frob}{\mathrm{Frob}}
\newcommand*{\ZZ}{\mathbb{Z}}

\def\Sur{{\rm Surj}}
\def\Tor{{\rm Tor}}

\def\Pic{{\rm Pic}}
\def\GL{{\rm GL}}
\def\Sp{{\rm Sp}}
\def\GSP{{\rm GSp}}

\def\End{{\rm End}}
\def\coker{{\rm Coker}}
\def\Hom{{\rm Hom}}
\def\ker{{\rm Ker}}
\def\im{{\rm Im}}
\def\avg{{\rm Avg}}
\def\Conf{{\rm Conf}}
\usepackage{amscd,amssymb,amsmath}
\usepackage{color}

\author{Michael Lipnowski, Jacob Tsimerman}
\begin{document}
\title[Cohen Lenstra for \'{E}tale Group Schemes]{Cohen Lenstra Heuristics for \'{E}tale Group Schemes and Symplectic Pairings}
\maketitle
\begin{abstract}

We generalize the Cohen-Lenstra heuristics over function fields to \'{e}tale group schemes $G$  (with the classical case of abelian groups corresponding to constant
group schemes). By using the results of Ellenberg-Venkatesh-Westerland, we make progress towards the proof of these heuristics. Moreover, by keeping track of the image of the 
Weil-pairing as an element of $\wedge^2G(1)$, we formulate more refined heuristics which nicely 
explain the deviation from the usual Cohen-Lenstra heuristics for abelian $\ell$-groups in cases where $\ell\mid 
q-1$;  the nature of this failure was suggested already in the works of Garton, EVW,  and others. On the purely large random matrix side, we provide a natural model which has the correct moments, and we
conjecture that these moments uniquely determine a limiting probability measure.

\end{abstract}

\section{Introduction}

In \cite{CL}, Cohen and Lenstra described a natural probability measure $m_{\mathrm{CL}}$ on the set of finite abelian $\ell$-groups; the Cohen-Lenstra measure of every finite abelian $\ell$-group $A$ is inversely proportional to $\# \Aut(A).$  The prediction that the distribution of $\ell$-parts of class groups of appropriate families of number fields is governed by this probability measure (suitably generalized) is known as the \emph{Cohen-Lenstra-Martinet Conjecture}.  Empirically, Cohen and Lenstra observed that $m_{\mathrm{CL}}$ correctly predicts the distribution of the $\ell$-part of class groups of quadratic fields, for $\ell$ an odd prime\footnote{If $\ell=2$, then Gerth conjectured that if one considers the subgroup of the Class group consisting of elements which are squares, the same conjecture holds. This `degenerate' case turns out to be more accessible and there are unconditional results in this direction by
Fouvry-Kulners \cite{FK}, Smith \cite{Smith},Milovic \cite{Milovic}, and Klys \cite{Klys}}. 

For functions $f$ defined on isomorphism classes of finite abelian $\ell$-groups which are absolutely integrable with respect to $m_{\mathrm{CL}},$ define 
$$\mathbf{E}_n(f) := \lim_{X \to \infty} \frac{\sum_{\deg(K / \Q) = n,|\mathrm{disc}(K)| < X} f(\mathrm{Cl}(K))}{\sum_{\deg(K / \Q) = n, |\mathrm{disc}(K)| < X} 1},$$
assuming the above limit exists.  Despite much work on the Cohen-Lenstra heuristics, the only unconditional results known to date are:
\begin{itemize}
\item
Davenport-Heilbronn's determination of the average size of the 3-torsion subgroup of $\mathrm{Cl}(K)$ for quadratic fields $K$ \cite{DH}:
$$\mathbf{E}_2( \# \mathrm{Surj}(\bullet, \ZZ / 3)) = \mathrm{Expectation}_{m_{\mathrm{CL}}}(\# \mathrm{Surj}(\bullet, \ZZ / 3))$$

\item
Bhargava's determination of the average size of the 2-torsion subgroup of $\mathrm{Cl}(K)$ for cubic fields $K$ \cite{Bhargava}: 
$$\mathbf{E}_3( \# \mathrm{Surj}(\bullet, \ZZ / 2)) = \mathrm{Expectation}_{m_{\mathrm{CL}}}(\# \mathrm{Surj}(\bullet, \ZZ / 2))$$ 
\end{itemize} 

Remarkable recent progress toward the Cohen-Lenstra Conjecture has been made for class groups of functions fields of curves over finite fields. In this case, using the methods of \'{e}tale cohomology and by proving results on homological stability, Ellenberg, Ventakesh, and Westerland have obtained unconditional results essentially proving that, for every finite abelian group $\ell$-group $A,$ the expectation $\mathbf{E}_2(\# \mathrm{Surj}(\bullet,A))$ is very close to 
$\mathrm{Expectation}_{m_{\mathrm{CL}}}(\# \mathrm{Surj}(\bullet,A))$  

To get a handle in the function field case, one looks at all the geometric $\ell$-power torsion points of the Jacobian
as a module for the Frobenius operator, of which the class group becomes one small piece.  As such, one of our main goals in this paper is to generalize these heuristics in the function
field case by remembering the entire action of the Frobenius operator. A convenient language for making this precise is that of \'{e}tale group schemes.

\subsection{\'{E}tale group schemes}

Let $C/\F_q$ be a (smooth, projective, irreducible) curve over a finite field. The class group of $C$ is naturally $J(\F_q)$ where $J$ is the Jacobian of $C$. One is then naturally
led to ask: what the distribution of $J(\F_q)[\ell^{\infty}]$ as $C$ varies in some natural family\footnote{One often takes some Hurwitz scheme for the family $C$ varies in, as an analogue 
of looking at number fields of fixed degree. It seems plausible one could look at all curves of a fixed genus, but this seems very difficult, and it is not clear what behaviour to expect.}?

The group $J(\F_q)[\ell^{\infty}]$ equals the kernel of $1-F$ acting on $J(\ol{\F_q})[\ell^{\infty}],$ where $F$ denotes the Frobenius operator.  The $\ell$-part of the class group is thus identified as the `1-eigenspace' of the Frobenius operator. However, there is no need to restrict attention to this particular eigenspace; we could consider all eigenspaces at once. More generally,
we consider monic polynomials $P(x)\in\Z_{\ell}[x]$ for which $P(0)$ is invertible and consider the kernel of $P(F).$  In fact, since the Jacobian is self-dual, all ``$\ker P(F)$ information" is contained in the cokernel of $P(F)$ on the Tate module
of the Jacobian, and the latter is the object we actually study.\footnote{Note that one issue which arises now is that the kernel could be infinite, but this should arise very infrequently, so that the distribution we obtain should be supported on
finite modules.}  This cokernel is a module with the action of a Frobenius operator. Moreover, because the Frobenius operator acts invertibly, the cokernel may be thought of as an \'{e}tale group 
scheme, which is the quotient of $J[\ell^n]$ by $P(F)$, for any large enough $n$. By using the results of \cite{EVW} we obtain information about the distribution of these group schemes.

\subsection{Symplectic pairings and the Cohen-Lenstra-Martinet heuristics}

In the case where the base number field contains roots of unity, Malle presented computational
evidence which yielded doubt on the Cohen-Lenstra-Martinet heuristics. Malle refined these heuristics, giving a different random model involving the symplectic group, and there has been much evidence that 
Malle's refinement is correct \cite{Achter} \cite{Garton}. We present a refinement Cohen-Lenstra heuristics in all cases, which we believe nicely explains these
discrepancies. 

The Weil-pairing on $J[\ell^n]$ can be thought of as a section of a certain naturally defined group scheme $\wedge^2 J[\ell^n](1)$ which we define in \S4. This pushes
forward, so we naturally get an element of $\wedge^2(J[\ell^n]/(1-F))(1)$. We thus can speak of a distribution not just on groups $A$, but on pairs $(A,\omega)$ where $\omega$ is an
element of $\wedge^2A (1)(\F_p)$. Now it turns out that if $\ell\nmid q-1$ then $\omega$ is forced to be 0, and hence we revert to the usual Cohen-Lenstra measure. However, if
$\ell\mid q-1$, then we in fact get a more refined distribution. In fact, it is natural to combine this decoration with the generalization to arbitrary \'{e}tale group schemes and this is what we carry out.

\subsection{Results}
$\newline\newline$
Our main result is as follows. See \S4 for precise definitions.

\begin{thm}[Corollary \ref{momentsare1general}]\label{main}
Let $\ell$ be an odd prime. Let $G$ be a finite \'{e}tale group scheme over $\F_q$ of order $\ell^n$,  and $\omega_G\in(\wedge^2 G)(1)(\F_q)$. 
For each $g$,  let $\avg(G,\omega_G,g,q)$ denote the average number of surjections from $\Pic^0(C)[\ell^n]$ to $G$ which push-forward the Weil-paring to $\omega_G$, where $C$
varies over hyperelliptic curves of genus $g$.

Let $\delta^{\pm}(q,\omega_G)$ be the lower and upper limits of $\avg(G,\omega_G,g,q)$ as $g\ra\infty$. Then as $q\ra\infty$ and $n$ stays fixed, $\delta^{+}(q,\omega_G)$
 and $\delta^{-}(q,\omega_G)$ converge to 1.
\end{thm}

Our proof of this theorem closely follows the strategy of \cite{EVW}. We represent the averages in question in terms of points on a moduli space we construct. These
moduli spaces turn out to be twists of the moduli spaces that appear in \cite{EVW}. We can therefore directly apply their results on cohomology bounds, and the theorem
follows from the Lefschetz trace formula once we identify the number of connected components of these moduli spaces. 

In \S2 and \S3 we develop foundational results on Cohen-Lenstra measures in the context of our decorated \'{e}tale group schemes. We obtain the strongest results
in the case where $\omega_G\in\wedge^2G(1)$ is forced to be 0 -- which happens `generically':

\begin{thm} [Theorem \ref{uniquenesswithoutsymplecticstructure}]
Let $P(x)\in\Z_{\ell}[x]$ be a monic polynomial, such that $P(q)$ is not divisible by $\ell$, and assume that $\ell$ is odd. Let $R=\Z_{\ell}[x]/P(x)$. There exists a unique probability 
measure $\mu$, supported  on finite $R$-modules, such that for any finite $R$-module $M$,  the expected number of surjections from a $\mu$-random module to $M$ is $1$.
Moreover, $\mu$ is supported on precisely the modules of projective dimension 1, and assigns such a module $M$ measure $\mu(M)=\frac{c}{\#\Aut(M)}$ where 
$c=\prod_{k_j}\prod_{i=1}^{\infty}(1-|k_j|^{-i})$ and the product is over the finite residue fields of $R$. 

\end{thm}

As a consequence of these Theorems, we obtain in proposition \ref{approximatemomentsimpliesapproximatelyCL} similar results
on limiting measures for our decorated \'{e}tale group schemes.

As a concrete application of our methods, we prove the following result on the independence of the class group of its hyperelliptic curves  and its quadratic twist.

\begin{thm}[Proposition \ref{independence}]
Suppose $\ell \nmid q^2 - 1$ and that $\ell \neq 2.$  Fix $\epsilon > 0.$  Fix a finite set $S$ of finite abelian $\ell$-groups. For a curve $C$ over $\F_q$, denote by 
$C^{\sigma}$ the quadratic twist of $C$.

There exists $Q(S)\gg 0$ such that if $q,g \geq Q(S)$ and $A,B \in S,$

$$\left| \mathrm{Prob} \left( \mathrm{Jac}(C)(\F_q)_\ell \cong A \text{ and } \mathrm{Jac}(C^{\sigma})(\F_q)_\ell \cong B \right) - \frac{c_R}{\# \mathrm{Aut}_{\Z_{\ell}}(A)\#\mathrm{Aut}_{\Z_{\ell}}(B)} \right| < \epsilon,$$

where $C$ varies over hyperelliptic curves of genus $g$, and $c_R$ is the normalizing constant from Theorem \ref{convmeas}.  
i.e. the class groups $\mathrm{Jac}(C)(\F_q)_\ell$ and $\mathrm{Jac}(C^{\sigma})(\F_q)_\ell$ behave almost independently for $g$ sufficiently large.
\end{thm}

\subsection{Plan of the paper}
\begin{itemize}

\item In \S2 we present a generalization of the usual Cohen-Lenstra measure to rings which are finite over $\Z_{\ell}$. 
\item In \S3 we construct a random model for pairs $(G,\omega\in \wedge^2(G)(1))$
where $G$ is a module over a ring $\Z_{\ell}[F]/P(F)$. We conjecture that our model yields a unique measure with a `moments equal 1' property, and using our
results in \S2 we prove this uniqueness in the case where $\omega$ is 
forced to be 0 -- in other words, when we don't have to keep track of any symplectic structure, so we can revert to a linearized model. 
\item In \S4 we use the work of 
Ellenberg-Venkatesth-Westerland to prove results analogous to theirs in the direction of Cohen-Lenstra for function fields, for our refined distributions.
\item  In \S5 we present some applications, notably to
the independence of the $\ell$-part of the class group of a hyperelliptic curve and its quadratic twist. 
\end{itemize}

\subsection{Acknowledgements} 
It is a pleasure to thank Krishnaswami Alladi for explaining how to prove the Cohen-Lenstra identities using iterated Durfee's squares.

\section{Large random matrices over rings}

\subsection{Summary}
The purpose of this section is to generalize the Cohen-Lenstra measure for finite abelian $\ell$-groups to the case of finite $R$-modules for certain rings $R$, finite over $\Z_{\ell}$. This
measure has the nice property that for every finite $R$-module $M$, the expected number of $R$-module surjections to $M$ is 1. \emph{The support of this measure is not
full}, but on the support the measure of $M$ is proportional to $\frac{1}{\#\Aut M}$.

\subsection{The Cohen-Lenstra Measure for $R$-modules}

Let $R$ be a finite, local $\Z_p$-algebra, with residue field $\F_R$ such that $\Z_p\subset R$.
 Let $S_R$ be the set of all finite $R$-modules and define a measure $\mu_{R,N}$ on $S_R$ as follows: 
Let $\phi_N:\End_R(R^N)\ra S_R$ be defined by $G\ra\coker G$. Then $\mu_N$ is the pushforward of Haar measure under $\phi_N.$ 
Recall that since $R$ is a local ring all projective modules are free.
Recall also that we say that a module $M\in S_R$ has projective dimension $1$ if it has a projective (free) resolution of length $1$: $0\ra F_1\ra F_2\ra M\ra 0$.
Call $T_R$ the set of modules $M\in S_R$ which occur in the image of $\phi_N$ for some $N$. Note that if $R$ is torsion free, $T_R$ coincides with the set of
modules of projective dimension 1.

We can give a simple homological criterion for a finite module $M$ to occur in $T_R$. Define $d_M=\dim_{\F_R} \Tor_R^1(M,\F_R)-\dim_{\F_R} M\otimes_R \F_R$.

\begin{lemma}

For all finite modules $M$, $d_M\geq 0$, and a module $M$ occurs in $T_R$ iff $d_M= 0$.

\end{lemma}

\begin{proof}

For all finite $M$, we can find a surjection $R^N\ra M$. Thus, we have an exact sequence $0\ra U\ra R^N \ra M$. Tensoring with $\F_R$ and taking the associated long exact sequence, we see that $\dim_{\F_R}\Tor_R^1(M,\F_R)-\dim_{\F_R} M\otimes_R \F_R = \dim_{\F_R}(U\otimes_R \F_R) - N$. Since $M[1/p]=0$, $U$ has $R[1/p]$-rank equal to $N$, and so a minimal
generating set for $U$ consists of at least $N$ elements, which means $\dim_{\F_R}(U\otimes_R \F_R) - N\geq 0$ by Nakayama's lemma. This shows $d_M\geq 0$. 

Now, if $M\in T_R$, then  we can find an exact sequence $0\ra K\ra R^N\ra R^N\ra M\ra 0$ for some $R$-module $K$. Thus we can take the $U$ in the above paragraph to be
$R^N/K$, and thus be generated by $N$ elements. Thus, in this case, $d_M=0$.

Conversely, if $d_M=0$, then the $U$ in the first paragraph must be generated by $N$ elements and so is a quotient of $R^N$. Thus $M\in T_R$. 

Now, if $R$ is torsion free, then as already mentioned $T_R$ coincides with the set of finite modules of projective dimension 1, which is equivalent to $\Tor_R^2(M,F)=0$. 
\end{proof}

\begin{remark}

We point out that another natural construction of $R$-modules -- at least in the case $R=\Z_p[F]/(P(F)) $ -- is as follows: One can take a random map $A\in\End(\Z_{p}^{d})$, and consider
$\coker P(A)$ as a module over $R$, with $F$ acting as $A$. In fact, this more directly mirrors what occurs in the geometric cases we consider, where $\Z_p^{2g}$ occurs
as a Tate module and $A$ as the Frobenius endomorphism. This turns out to be more difficult to study, which is why we focus on the model we have presented. However, one can realize
$\coker P(A)$ in our context as the Cokernel of $F-A$ acting on $R^d$, since $$R^d/(F-A) \cong \Z_p[F]^d/(F-A,P(F)) = \Z_p[F]^d/(F-A,P(A)) = \Z_p^d/(P(A)).$$ 

\end{remark} 
\begin{thm} \label{convmeas}

The $\mu_{R,N}$ converge (in the weak-* topology) to a probability measure $\mu_R$, supported on $T_R$, such that for $M\in S_R$ we have 
$\mu_R(M)= \frac{c_R}{|\Aut_R(M)|}$,where $c_R = \lim_{n\ra\infty} \frac{|\GL_n(\F_R)|}{|M_n(\F_R)|} =\prod_{i=1}^{\infty}(1-|\F_R|^{-i}).$

\end{thm}

\begin{proof} Let $M$ be an $R$-module. If $M$ is not in $T_R$ then by definition $M$ never occur as the cokernel of an endomorphism $G$ and thus
cannot be in the support of $\mu_{R,N}$ for any $N$. So wlog $M\in T_R$. 

Now let us compute $\mu_{R,N}(M)|\Aut_R(M)|$. Consider the space $M_N(R)\times M^N$, with a choice of Haar measure giving total measure $|M|^N$. We can identify $M^N$ with $\Hom(R^N,M)$. Now consider the subset $X$  consisting of $(G,\phi)$ so that $\im(G)=\ker\phi$. The set of all such $G$ such that $\coker G\sim M$ has measure $\mu_N(R)$, and for each such $G$ there are $\Aut_R(M)$ choices of $\phi$ certifying the isomorphism. Thus, the measure of $X$ is $\mu_{R,N}(M)|\Aut_R(M)|$.

We now compute the measure of $X$ in a different way, by fibering over $\phi$ instead.  Now, since $M\in T_R$ there is an exact sequence
$$R^a\xrightarrow{g} R^a\xrightarrow{f} M\ra 0.$$ Let $C_f$ be the kernel of $f$.  Now, the number of maps from $R^N$ to $M$ is $|M|^N$, and
with probability tending to $1$ as  $N\ra\infty$, a random such map  $\phi$ is a surjection. Moreover, with probability tending to $1$ some subset of size $a$ of the co-ordinates 
induces the map $f:R^a\ra M$. Whenever this happens, we may make a unipotent change of co-ordinates so that the other $N-a$ co-ordinates all map to $0$,
and thus the kernel is isomorphic to $C_f\oplus R^{N-a}$. Now the measure of all $G$ whose image is contained in $\ker\phi$ is $|M|^{-N}$. Of those, we need to find the 
measure of those maps that give a surjection. We thus need to compute 
$\frac{\mu_{haar}(\textrm{Surj}(R^N,C_f\oplus R^{N-a})}{\mu_{haar}(\Hom(R^N,C_f\oplus R^{N-a})}$. By Nakayama's lemma, it is sufficient to tensor everything with
the fraction field $\F_R$ of $R$, so we are reduced to showing that $C_f\otimes_R\F_R \sim \F_R^a$. Since $C_f$ is a quotient of $R^a$ it can be 
generated by at most $a$ elements. Moreover, as can be seen by tensoring with $\Q_p$, there can be no fewer than $a$ elements in a generating set for 
$C_f$. By Nakayama's lemma again, we see that $C_f\otimes_R \F_R \sim \F_R^a$ as desired.

All that remains is to show that the $\mu_{R}$ is really a probability measure (i.e there is no escape of mass). Note that the above argument shows that
$\mu_{N,R}(M)\leq \frac{1}{\Aut_R(M)}$. Since $\mu_{R}$ has $L^1$-norm at most 1, for any $\epsilon>0$ we can pick a co-finite set $S\subset T_R
$ so that $\sum_{M\in S}\frac{1}{|\Aut_R(M)|}<\epsilon/2$ and large enough $N$ so that $\sum_{M\not\in S}|\mu_{N,R}(M)-\mu_R(M)|<\epsilon/2$, from which
it follows that $\mu_R$ has $L^1$-norm at least $1-\epsilon$. The result follows.

\end{proof}

We can also compute the moments of the measure above. As expected by analogy to classical Cohen-Lenstra, they are all equal to 1.

\begin{prop}

For any finite module $M_0$, $$\sum_{M\in S_R} \#\Sur(M,M_0)\mu_R(M)=1.$$

\end{prop}

It is worth remarking that we do not insist in the above proposition that $M_0\in T_R$.

\begin{proof}

Fix an $N<0$. Then letting $\mu_{haar}$ be the Haar measure on $\End_R(R^N)$ giving total measure $1$, we see that
\begin{align*}
\sum_{M\in S_R} \#\Sur(M,M_0)\mu_{R,N}(M) &= \int_{\phi\in\End_R(R^N)} \#\Sur(\coker\phi, M_0) d\mu_{haar}\\
&= \sum_{\psi\in \Sur(R^N,M_0)}\mu_{haar}(\phi\mid \coker\phi \in \ker\psi)\\
&= \#\Sur(R^N,M_0)\cdot |M_0|^{-N}\\
\end{align*}

Now, as $N\ra\infty$, $\#Sur(R^N,M_0)\sim |M_0|^N$. Thus, $$\lim_{N\ra\infty}\sum_{M\in S_R} \#\Sur(M,M_0)\mu_{R,N}(M)=1.$$
By the proof of the theorem above, $\mu_{R,N}\leq c_R^{-1}\mu_R$, so the sum converges absolutely, and the result follows.

\end{proof}

We expect that the moments actually determine our measure $\mu_R$. We expect this to be true in all cases, though we cannot show it in the case that 
$\F_R=\F_2$ though we expect it to be true in this case also. The proof is identical to \cite[Lemma 7.2]{EVW}, but we write it anyways.

\begin{lemma}\label{unique}

Assume that $\F_R\neq \F_2$. If $\mu$ is any measure  on $S_R$  such that the expected number of surjections from a $\mu$-random module to $M_0$ is $1$ for any 
finite $M_0$, then $\mu=\mu_R$. The same conclusion holds for $\mu$ being any function in $L^1(S_R)$.

\end{lemma}

\begin{proof}

Consider the infinite dimensional operator $U$ on $L^{\infty}(S_R)$ such that $U_{M,M'} =\#\Sur(M,M')/\#\Aut(M)$. Now, the rows of $U$ have sums  $c_R^{-1}$
and so $U$ is indeed an operator on $L^{\infty}(S_R)$. Moreover, the elements of $U-1$ are positive and have row sums 
$c_R^{-1}-1$. Now, to estimate $c_R$, let $q=|\F_R|^{-1}$ and note that by the Euler identity we have 
$$c_R=\sum_{n\in\Z}(-1)^nq^{(3n^2-n)/2} > 1-\frac{q}{1-q} = \frac{1-2q}{1-q}. $$ Since $q\leq \frac13$ by assumption, we conclude that $c_R>\frac{1}{2}$. Thus, the norm of $U-1$ is 
less than 1 and so $U$ is invertible  with inverse $\sum_{j=0}^{\infty} (1-U)^j$. 

Now, consider the $L^1$-function $\mu$. Since the moments to any $M_0\in S_R$ is $1$, we must have that $\mu(M_0)\leq \#\Aut(M)^{-1}$.  
Thus the vector $V$ with $V_M=\mu(M)\Aut(M)$ is in $L^{\infty}(S_R)$. Now the condition on the moments of $\mu$ amounts to saying that $UV=1_{S_R}$. 
Thus, we must have that $V=U^{-1}1_{S_R}$, and so there is a unique such function $\mu$, which must then be $\mu_R$. 

\end{proof}

\subsection{Remarks on identities}

We give an example of an $R$ with torsion where $T_R$ is larger than the set of modules with projective dimension 1. Consider $R=\Z_p[x]/(px,x^2)$.
Take $M=R/pR$. Clearly $M$ occurs in $T_R$. On the other hand, if $M$ had projective dimension 1 then $pR$ would be forced to be projective, and thus
free since $R$ is local. However, $pR$ is annihilated by $x$, and thus cannot be free.
In fact, $M$ fits into the exact sequence $0\ra\F_p\ra R\ra R \ra M$. Since $R$ is not a regular local ring, $\F_p$ has infinite projective dimension, and thus $M$ does
as well. 

For such rings $R$, if we instead considered the measure arising from the cokernel of a map $R^{N+d}\ra R^N$ we could 
conceivably get more and more modules $M$ in the support, giving a range of identities.  They would be more and more complicated, however:
For a module $M$, let $d_M=\dim_{\F_R} \Tor_R^1(M,\F_R)-\dim_{\F_R} M\otimes_R \F_R$. Then we get (by a minor modification of the proof above) the following identities:
$$\sum_{d_M\leq d} \frac{\prod_{j=1}^{d-d_M}(1-|\F_R|^{-i})^{-1}}{|M|^d|\#\Aut_R(M)|} = c_R^{-1}.$$

In fact, we can derive a series of \emph{finite} identities from the above. Consider again $R=\Z_p[x]/(px,x^2)$. Then $R$ maps to $\Z_p$, and it is easy to see from the construction that
$\mu_R$ pushes forward to $\mu_{\Z_p}$. Moreover, and $R$-module $M$ maps to $M/xM$, and it is easy to show by row and column operations and the fact that $x$ is nilpotent that 
$M$ is bounded in terms of $d_M$ and $M/xM$. Thus, we conclude that for each $p$-group $A$, we have

$$\sum_{d_M=0,M/xM\sim A} \frac{1}{\#\Aut_R(M)} = \frac{1}{\#\Aut_{\Z_p} (A)}.$$

Of course, one can generalize this to arbitrary local maps $R\ra S$ (though perhaps one has to be a bit careful if one wants the sum to remain finite). 
It is not clear to us, even for the above identity, how to prove it by elementary means.

%
%
%
%

\section{A random model for \'{E}tale group schemes with a Symplectic form}
%
%
%

\subsection{\'{E}tale group schemes and symplectic forms}

Let $q$ be a prime power and $\ell$ a prime not dividing $q$. Let $P(x)\in\Z_{\ell}[x]$ be a monic polynomial satisfying $\ell\nmid P(0)$. Consider the collection $\mathcal{E}$ of 
(isomorphism classes of) triples $(G,F_G,\omega_G)$ where $G$ is a finite 
abelian $\ell$-group, $\omega_G \in \wedge^2 G,$ and $F_G$ is an endomorphism of $G$ for which $P(F_G)=0, F_G(\omega_G) = q \omega_G.$  Note that since $P(F_G)=0$ it follows 
that $F_G$ is an automorphism.

$(G,F_G)$ functorially corresponds to an \'{E}tale group scheme $\mathcal{G}$  over $\F_q$ whose $\ol{\F_q}$ points are isomorphic to $G$ with the Frobenius action 
corresponding to $F_G$ . We shall construct in \S4 a natural group scheme $\wedge^2 \mathcal{G}$ and its Tate twist 
$\wedge^2\mathcal{G}(1)=\wedge^2\mathcal{G}\otimes \mu_{\ell^{\infty}}^{-1}$, whose $
\ol{\F_q}$
points naturally correspond to $\wedge^2 G$  (once one picks a section of $\mu_{\ell^{\infty}}(\ol{\F_q})$ ), and the Frobenius action is given by $q^{-1}F_G$. Thus, the set $\mathcal{E}$ 
naturally
corresponds to pairs $\left(\mathcal{G}, \omega\in(\wedge^2\mathcal{G}\otimes \mu_{\ell^{\infty}}^{-1})(\F_q)\right)$ where $\mathcal{G}$ is a finite \'{E}tale $\ell$-group scheme over $\F_q$. This is
our motivation for studying $\mathcal{E}$, as we are interested in constructing probability distributions on \'{E}tale group schemes , together with a section of $\wedge^2\otimes\mu_{\ell^{\infty}}^{-1}$. 

\subsection{Defining a measure}

We define $\GSP_{2g}^{(q)}$ to be the coset of $\Sp_{2g}$ in $\GSP_{2g}$ whose elements all scale the symplectic form by $q$.
For each positive integer $g$, there is a map
\begin{align*}
\GSP^{(q)}(\ZZ_\ell^{2g}, \omega) &\rightarrow \mathcal{E} \\
F &\mapsto\left(\coker(P(F)), F \mod P(F), \omega \mod P(F)\right)
\end{align*}
which gives rise to a probability measure $\mu_g$ on $\mathcal{E}$ by pushing forward the Haar probability measure on $\GSP_{2g}(\mathbb{Z}_\ell).$  

\begin{thm} \label{momentswithsymplecticstructure}
Fix $(H,F_H, \omega_H) \in \mathcal{E}$.  The $\mu_g$-expected number of equivariant surjections $T: (G,F_G) \twoheadrightarrow (H,F_H)$ for which $T(\omega_G) = \omega_H$ is equal to 0 for $g\leq g(H)$, and is equal to 1 for all $g>g(H,\omega_H)$ where $g(H,\omega_H)$ depends only on $H$.
\end{thm}

\begin{proof}
By definition of $\mu_g,$ the expected number of such surjections equals the $\GSP_{2g}^{(q)}(\ZZ_\ell)$-Haar expected number of surjections  $T: \coker(P(F)) \twoheadrightarrow H,$ for which $F$ induces $F_H,$ and for which $T(\omega) = \omega_H.$  Such surjections are equivalent to the following data: 
\begin{itemize}
\item
a surjection $T: \mathbb{Z}_\ell^{2g} \twoheadrightarrow H$ for which 
\begin{itemize}
\item
$T\circ P(F) = 0,$  
\item
$T F = F_H T,$ and
\item
$T \omega = \omega_H.$
\end{itemize}
\end{itemize}
The condition $T\circ P(F)=0$ is actually redundant; since $P(F_H) = 0,$ the second condition above implies that 
$$T\circ P(F) = P(F_H)\circ T =0.$$
Suppose $H$ is killed by multiplication by $\ell^n.$  Then the expected number of surjections equals
\begin{equation} \label{expectednumberofsurjections}
\frac{\# \{ (T,F): T: (\ZZ / \ell^n)^{2g} \twoheadrightarrow H, F \in \GSP^{(q)}((\ZZ / \ell^n)^{2g},\omega), T\circ F = F_H \circ T, T(\omega) = \omega_H \ \}  }{\#\GSP_{2g}^{(q)}(\ZZ / \ell^n) }.
\end{equation}
By an analogue of Witt's extension theorem \cite[Theorem 2.14]{Michael}, the image of the mapping $T \mapsto T(\omega)$ from surjections to symplectic forms on $H$ contains 
$\omega_H$ precisely for  $g>g(H,\omega_H)$, and 
for every $g$, every fiber forms a single orbit $\mathcal{O}$ under $\Sp(\ZZ_\ell^{2g},\omega)$ (where the symplectic group acts by precomposition).  Furthermore, suppose that $TF = F_H 
T$ and $T(\omega) = \omega_H.$  Then for every $g \in \Sp(\ZZ_\ell^{2g},\omega)$
$$(T g) (g^{-1} F g) = TF g = F_H (Tg)$$
and
$$Tg(\omega) = T\omega = \omega_H.$$
Therefore, among the pairs $(F,T)$ enumerated in the numerator of \eqref{expectednumberofsurjections}, the fibers over every $T$ have the same size.  Now assuming it exists, fix $T_0$ satisfying $T_0(\omega) = \omega_H$ (If no such $T_0$ exists, than the moment is clearly 0).   Then 
\begin{align} \label{stabilizersurjection}
&\frac{\# \{ (T,F): T: (\ZZ / \ell^n)^{2g} \twoheadrightarrow H, F \in \GSP^{(q)}((\ZZ / \ell^n)^{2g},\omega), TF = F_H T, T(\omega) = \omega_H \ \}  }{\#\GSP_{2g}^{(q)}(\ZZ / \ell^n) } \nonumber \\
&= \frac{\# \{ F \in \GSP^{(q)}((\ZZ / \ell^n)^{2g},\omega): T_0 F = F_H T_0  \}  }{\#\GSP_{2g}^{(q)}(\ZZ / \ell^n) } \cdot \# \mathcal{O} \nonumber \\
&= \frac{\# \{ F \in \GSP^{(q)}((\ZZ / \ell^n)^{2g},\omega): T_0 F = F_H T_0  \}  }{\#\GSP_{2g}^{(q)}(\ZZ / \ell^n) } \cdot \frac{\# \Sp((\ZZ / \ell^n)^{2g}, \omega)}{\# \mathrm{Stab}_{\Sp((\ZZ/\ell^n)^{2g},\omega)}(T_0)} \nonumber \\
&=  \frac{\# \{ F \in \GSP^{(q)}((\ZZ / \ell^n)^{2g},\omega): T_0 F = F_H T_0  \}  }{\# \{ g \in \Sp((\ZZ / \ell^n)^{2g},\omega): T_0 g = T_0 \}  }
\end{align}
The set in the numerator of \eqref{stabilizersurjection} is either empty or is a torsor for the group in the denominator.  Thus we only need to show that a single such $F$ exists.

Now, to show this, consider first any element $F_0\in \Sp((\ZZ/\ell^n)^{2g},\omega)$. Then $F_H^{-1}T_0F_0$ is a surjection from 
$((\ZZ/\ell^n)^{2g},\omega)$ to $(H,\omega_H)$. Thus, by \cite[Theorem 2.14]{Michael}, there exists an element $g\in\Sp((\ZZ/\ell^n)^{2g},\omega)$ satisfying
$$F_H^{-1}T_0F_0g=T_0$$ and therefore $T_0F_0g=F_HT_0$. Thus we may take $F=F_0g$ and this completes the proof.
\end{proof}

\subsection{The existence of a limit measure}

In light of the results of the previous section, and analogous results for the Cohen-Lenstra measure (\cite{EVW},\cite[Theorem 8.2]{Wood}) we make the following conjecture, which roughly says that
the moments constitute enough information to recover the full measure in cases of interest:

\begin{conj}\label{limit}

The measures $\mu_g$ converge to a measure $\mu$ on $\mathcal{E}$, such that the expected number of surjections from a $\mu$-random element to any element in 
$\mathcal{E}$ is 1. Moreover, this property characterizes $\mu$.

\end{conj}

We devote the rest of this section to proving conjecture \ref{limit} in a couple special cases. Most notably, we can use the results of \S2 to prove the conjecture in the case where the 
symplectic structure `doesn't come up'. In that case we can use the much easier additive model in \S2 as opposed to the model with symplectic matrices. We `get rid of' the
symplectic structure as follows: If $P(q)$ is not divisible by $\ell$, then since $\omega_G$ is killed by both $P(q)$ and a power of $\ell$ it is forced to be 0, so $\mathcal{E}$ is equivalent to the category of finite $\Z_{\ell}[x]/P(x)$ modules.


\begin{thm} \label{uniquenesswithoutsymplecticstructure}
In the notation above, assume that $P(q)$ is not divisible by $\ell$, and assume that $\ell$ is odd. Then conjecture \ref{limit} holds. Moreover, $\mu$ is supported on precisely the 
$R=\Z_{\ell}[x]/P(x)$ modules of projective dimension 1, and assigns such a module $M$ measure $\mu(M)=\frac{c}{\#\Aut(M)}$ where $c=\prod_{k_j}\prod_{i=1}^{\infty}(1-|k_j|^{-i})$ 
and the product is over the finite residue fields of $R$. 

\end{thm}

\begin{proof} 
Assume first that $R$ is a local ring. Then note that we have already constructed one measure satisfying the above hypothesis on moments in Theorem \ref{convmeas}, and Lemma
\ref{unique} guarantees that these moments specify a unique measure. Hence it is sufficient to show that the $\mu_g$ converge to a measure with 1 expected surjection to each finite $R$
module. Note that by Theorem \ref{momentswithsymplecticstructure} the $\mu_g$ satisfy $\mu_g(M)\leq\frac{1}{\#\Aut M}$ for any finite $R$-module $M$. Letting $S_R$ be
the set of finite $R$-modules as in \S2, there is an operator $U$ on $L^{\infty}(S_R)$ given by $U_{M,M'} = \frac{\#\Sur(M,M')}{\#\Aut(M)}$. Now the vector $V_g\in L^{\infty}(S_R)$ given by
 $V_M=\mu_g(M)\#\Aut(M)$ has $L^{\infty}$ norm bounded by 1. Further, by Theorem \ref{momentswithsymplecticstructure} the product $UV_g$ is a vector $W_g$ consisting 
 of $0$ and $1$ entries, whose entries each eventually become $1$ as $g$ increases. Thus, we can write $V_g=U^{-1}(W_g)$. Since as in Lemma \ref{unique} the operator $U^{-1}$ is bounded, it can be 
 represented as an infinite matrix with rows in $L^1$ with uniformly bounded $L^1$ norm. Since the $W_g$ have $L^{\infty}$ norm bounded by $1$ and each entry eventually stabilizes,
 we can conclude that the $\mu_g$ converge to $\mu$ in the weak-* topology, as desired.
 
 Now, even if $R$ is not local, it is $\ell$-adically complete and $R/\ell$ is artinian, so $R$ is a product of local rings $R=\prod_j R_j$. It follows that we can take 
 $\mu=\prod_{j} \mu_{R_j}$ and this measure will have all the correct moments, and the exact same proof as in the previous paragraph shows that the $\mu_g$ converge to $\mu$. Thus it 
 only remains to show that $\mu$ is determined by its moments. Note here that the exact same proof as in the local case won't work, since the constant $c$ could be less than $\frac{1}{2}$.
Going by induction on the number of local rings that $R$ is a product of, we may write $R=R_1\times R_2$ where $R_{1,2}$ have the property that the corresponding measures 
$\mu_{R_{1,2}}$ are determined by their moments. Now, suppose that $m$ is any other measure with the correct moments. 
For an $R_1$-module $M_1$ and an $R_2$-module $M_2^0$ we 
 let $$a(M_1,M^0_2):=\sum_{M_2}m(M_1\times M_2)\#\Sur(M_2,M^0_2).$$
 
 Then it follows that for each $M_2^0,M_1^0$, that $$\sum_{M_1} a(M_1,M^0_2)\#\Sur(M_1,M_1^0)=1.$$ Thus, by our induction assumption for $R_2$ it follows that 
 $a(M_1,M_2)=\mu_{R_1}(M_1)$.  Now, by our induction assumption for $R_1$ we learn that $$m(M_1\times M_2) = \mu_{R_1}(M_1)\times\mu_{R_2}(M_2)=\mu(M_1\times M_2)$$ 
 as desired. 

\end{proof}

In the case where the symplectic structure is present, we don't even have a good conjecture as to what the limiting measure  in conjecture \ref{limit} should be. It is natural to guess that
it is proportional to the inverse of the size of the automorphism group - where now one only takes automorphisms if they preserve $\omega_G$ - but this does not agree with computations 
of Garton\cite{Garton}! We think it would be very interesting to at least develop a plausible heuristic.

%
%

\subsection{Moments approximately 1 implies approximately Cohen-Lenstra} \label{momentsapproximately1}
Fix a finite subset $S' \subset \mathcal{E}.$  Let $\mathrm{Conf}_g$ denote the moduli space of $n$ distinct, unordered unlabelled points in $\mathbb{A}^1.$  Let $\mathcal{C} \rightarrow \mathrm{Conf}_g(\mathbb{A}_1)$ denote the associated family of hyperelliptic curves.  Let $F_x$ denote Frobenius acting on the $\ell$-adic Tate-module of $\mathrm{Jac}(C_x).$  For every $g,$ let $\nu_g$ be the discrete probability measure
$$\nu_g = \frac{1}{\# \mathrm{Conf}_g(\F_q)} \sum_{x \in \mathrm{Conf}_g(\F_q)} \delta_{\coker P(F_x)}.$$

Building on the work of \cite{EVW}, we will show in \S \ref{geometry} that for any $\delta > 0,$ there is some $Q(S',\delta), G(S',\delta) \gg 0$ such that for all $q \geq Q(S,\delta)$, provided $g \geq G(S,\delta),$ then
\begin{equation} \label{geometryapproximatemoments}
\mathrm{Expectation}_{\nu_g}(\# \mathrm{Surj}(\bullet, A)) \in [1-\delta, 1 + \delta] \text{ for all } A \in S'.
\end{equation} 

Geometry gives us access to moments, and we would like to recover as much information about the measures $\nu_g$ as we can from a large set of approximate moments as in \eqref{geometryapproximatemoments}.

\begin{definition} \label{subexponentialenlargements}
Let $R = \ZZ_\ell[x] / P(x).$  Let $A$ be a finite $R$-module.  An \emph{enlargement} $A'$ of $A$ is an $R$-module admitting a surjection onto $A$ whose kernel is a simple $R$-module.  An \emph{$s$-enlargement} $B$ of $A$ is a finite $R$-module admitting a surjection onto $A$ whose kernel has $R$-length equal to $s.$  

Say that $R$ has the \emph{few enlargements property} if for every finite $R$-module $A,$ the number of isomorphism classes of $s$-enlargements of $A$ is subexponential in $s.$
\end{definition}

\begin{lem} \label{productshavefewenlargements}
If $R$ is a product of maximal orders, then $R$ satisfies the few enlargements property.
\end{lem}
\begin{proof}  
This follows exactly as in the argument from \cite[Lemma 8.4]{EVW}. 
\end{proof}

\begin{prop} \label{approximatemomentsimpliesapproximatelyCL} 
Suppose that $\ell$ does not divide $P(q),$ and  $R = \ZZ_\ell[x] / P(x)$ has the few enlargements property.  Let $\nu$ be a probability measure on $\mathcal{E}.$  Fix a finite subset $S \subset \mathcal{E}.$  Fix $\epsilon > 0.$  There exist $\delta > 0$ and a finite subset $S' \subset \mathcal{E}$ satisfying: 
\begin{align*}
\mathrm{Expectation}_\nu(\# \mathrm{Surj}(\bullet, A')) &\in [1-\delta,1 + \delta] \text{ for all } A' \in S' \\
\implies | \nu(A) - \mu_R(A)| &< \epsilon \text{ for all } A \in S.
\end{align*}
\end{prop}

\begin{proof}
The hypothesis $\ell \nmid P(q)$ ensures that the symplectic form equals 0.  The argument from \cite[Proposition 8.3]{EVW} carries over verbatim to the present context.
\end{proof}

\section{Moments of \'{E}tale Group Schemes via the Lefschetz Trace Formula} \label{geometry}

In this section we define moduli spaces over $\F_q$, whose $\F_q$-points correspond to surjections from torsion subgroups of 
Jacobians of hyperelliptic curves to \'{e}tale group schemes $\mathcal{G}$ together with a section of $\wedge^2\mathcal{G}(1)$, and prove Theorem \ref{main}. In particular, we identify
the rationally defined geometric components of the moduli spaces considered in \cite{EVW} with the set $\wedge^2\mathcal{G}(1)(\F_q)$

\subsection{Multilinear algebra for \'{E}tale group schemes}
Let $S$ be a scheme.  Let $G / S$ be a finite \'{E}tale group scheme.

\begin{prop} \label{wedge2Etalegroupschemes}
Let $G / S$ be a finite, commutative \'{E}tale group scheme.  There exists a finite \'{E}tale group scheme $\wedge^2 G / S$ and a morphism $\iota: G \times_S G \rightarrow \wedge^2 G$ satisfying the following universal property: 
\begin{itemize}
\item[(a)]
$\iota$ is biadditive, i.e. for all $S$-schemes $T$ and all $x,y,z \in G(T),$
$$\iota(x + y,z) = \iota(x,z) + \iota(y,z) \text{ and } \iota(z, x+y) = \iota(z,x) + \iota(z,y).$$

\item[(b)]
$\iota$ is alternating, i.e. for all $S$-schemes $T$ and all $v \in G(T),$
$$\iota(v,v) = 0 \in \left( \wedge^2 G \right)(T).$$

\item[(c)]
The morphism $\iota$ is universal with respect to the properties (a),(b):  if $f: G \times_S G \rightarrow H$ is a biadditive, alternating morphism to commutative group scheme $H / S,$ there is a unique $S$-group scheme morphism $\pi$ for which $f = \pi \circ \iota.$  
\end{itemize}
\end{prop}

\begin{proof}
Let $ \underline{A}_S$ denote the constant group scheme on finite abelian group $B.$  Let $H / S$ be another group scheme.  A morphism $\underline{A}_S$ is determined by a collection of sections $s_a \in H(S)$ indexed by $a \in A$ satisfying $s_{a + b} = s_a + s_b$ for all $a,b \in A.$  Let $\underline{a} \in \underline{A}(S)$ denote the constant section determined by $a \in A.$  The morphism 
\begin{align*}
\underline{A}_S \times \underline{A}_S &\rightarrow \underline{\wedge^2 A }_S \\
(\underline{a},\underline{b}) &\mapsto \underline{a \wedge b}
\end{align*}   
is biadditive, altrnating, and satisfies the desired universal property by the universal property of $\wedge^2$ for abelian groups.

For more general finite \'{E}tale group schemes $G / S,$ the desired $\wedge^2 G$ may be constructed by descent.  Let $\{ U_{\bullet} \rightarrow S \}$ be an \'{E}tale cover trivializing the finite \'{E}tale group scheme $G.$  The above already constructs $\iota_{U_1}: G_{U_1} \times_{U_1} G_{U_1} \rightarrow \wedge^2 G_{U_1}$ and $\iota_{U_2}: G_{U_2} \times_{U_2} G_{U_2} \rightarrow \wedge^2 G_{U_2}.$  Then $\left( \wedge^2 G_{U_1} \right)_{U_1 \times_S U_2}$ and $\left( \wedge^2 G_{U_2} \right)_{U_1 \times_S U_2}$ both satisfy the universal property defining $\wedge^2 G_{U_1 \times_S U_2}.$  Thus, there is a unique isomorphism $\iota_{U_1,U_2}: \left( \wedge^2 G_{U_1} \right)_{U_1 \times_S U_2}$ and $\left( \wedge^2 G_{U_2} \right)_{U_1 \times_S U_2}$ commuting with the struction morphisms $(\iota_{U_1})_{U_2}$ and $(\iota_{U_2} )_{U_1}.$  A second application of the universal property shows that these isomorphisms satisfy the cocycle condition on triple overlaps.  By \'{E}tale descent, $\{\iota_{U_{\bullet}}: G_{U_\bullet} \times_{U_\bullet} G_{U_{\bullet}} \rightarrow \wedge^2 G_{U_{\bullet}}  \}$ descends to a biadditive, alternating morphism $\iota: G \times_S G \rightarrow \wedge^2 G.$  

Let $f: G \times_S G \rightarrow H$ be a biadditive, alternating map.  By the universal property, every $f_{U_\bullet} : G_{U_\bullet} \times_{U_\bullet} G_{U_\bullet} \rightarrow H_{U_\bullet}$ factors uniquely through $\wedge^2 G_{U_{\bullet}} \xrightarrow{\pi_{\bullet}} H_{U_{\bullet}}.$  By the universal property of $\wedge^2,$ the morphisms $\pi_{\bullet}$ must agree on double overlaps: $\iota_{U_1,U_2} \circ (\pi_1)_{U_1 \times_S U_2} = (\pi_2)_{U_1 \times_S U_2}.$  By \'{E}tale descent for morphisms, $\pi_{\bullet}$ descends uniquely to a morphism $\pi: \wedge^2 G \rightarrow H$ satisfying $\pi \circ \iota = f.$     
\end{proof}

A completely analogous argument allows one to make any tensorial construction for finite \'{E}tale group schemes.  The key point: universal properties from linear algebra induce descent data that allow one to \'{E}tale-localize the construction to the case of constant group schemes, for which the construction is simple.  We single out the following special case for later use:

\begin{prop} \label{homEtalegroupschemes}
Let $G_1 / S$ and $G_2 / S$ be finite commutative \'{E}tale group schemes.  There exists a finite commutative \'{E}tale group scheme $\Hom(G_1,G_2) / S$ equipped with a morphism $e: G_1 \times_S \Hom(G_1,G_2) \rightarrow G_2$ satisfying the following universal property: 
\item[(a)]
$e$ is biadditive, i.e. for all $S$-schemes $T$ and all $x,y \in G_1(T)$ and $\alpha,\beta \in \Hom(G_1,G_2),$ 
$$e(x + y, \alpha) = e(x,\phi) + e(y,\alpha) \text{ and } e(x, \alpha + \beta) = e(x,\alpha) + e(x,\beta).$$

\item[(b)]
The morphism $e$ is universal with respect to the property (a):  if $H/S$ is a commutative group scheme and $f: G_1 \times_S H \rightarrow G_2$ is biadditive, there is a unique $S$-group scheme morphism $\pi: \Hom(G_1,G_2) \rightarrow H$ for which $e = f \circ (1,\pi).$ 

Furthermore, $\Hom(G_1,G_2)$ represents the functor on $S$-schemes
$$T \mapsto \mathrm{Hom}_{T\text{-group schemes}}((G_1)_T, (G_2)_T).$$ 
\end{prop}

\subsection{Generalities on moduli spaces} \label{modulispaces}    

Let $V / S$ be a family of principally polarized, $g$-dimensional abelian varieties over $S.$  Let $\mathcal{A} \rightarrow V$ denote the universal family.  Suppose that $\ell$ is invertible on $S.$  Let $G / S$ be a finite \'{E}tale group scheme.  We claim the moduli problem

$$V_G(T) = \{  A \in V(T), T\text{-group morphism } \phi: A[\ell^n] \twoheadrightarrow G_T \} \text{ for all } T / S$$

is representable. To see this, consider the finite \'{E}tale group $S$-scheme $\Hom(\mathcal{A}[\ell^n], G_V)$, and consider the subscheme  $Y$ of 
$\mathcal{A}[\ell^n]\times \Hom(\mathcal{A}[\ell^n], G_V)$ mapping to the origin in $G_V$.   $Y$ is a finite \'{E}tale group scheme over $\Hom(\mathcal{A}[\ell^n], G_V)$, so $V_G$ is just 
the subscheme over which $Y$ is of degree $\ell^{2ng}/|G|,$ which is a union of connected components of $\Hom(\mathcal{A}[\ell^n], G_V)$. 

The morphism $V_G \rightarrow V$ is thus finite \'{E}tale. 


\subsubsection{The Weil pairing morphism to $\wedge^2 G \otimes \mu_{\ell^n}^{-1}$}
Let $T$ be an $S$-scheme.  Let $A / T$ be a principally polarized abelian $T$-scheme.  Let $\ell$ be invertible on $S.$  We may naturally regard the Weil pairing $w_{\ell^n}(A): A[\ell^n] \times_T A[\ell^n] \rightarrow \mu_{\ell^n} / T$ as an element of $\Hom\left(\wedge^2 A[\ell^n],\mu_{\ell^n} \right)(T).$ 

Finite \'{E}tale group schemes locally isomorphic to $\underline{\mathbb{Z} / \ell^n}_S$ form an abelian group under tensor product with identity $\underline{\mathbb{Z} / \ell^n}_S$ and inverse $H^{-1} := \Hom(H,\underline{\mathbb{Z} / \ell^n}_S).$ We let $H^m := H^{\otimes m}$ and $H^{-n} := (H^{-1})^{\otimes n}.$

Consider the multilinear map  
$A[\ell^n]^4\ra \mu_{\ell^n}\otimes\mu_{\ell^n}$ given on sections by 
$$(a,b,c,d)\ra w_{\ell^n}(a,c)\otimes w_{\ell^n}(b,d) \cdot \left[ w_{\ell^n}(b,c)\otimes w_{\ell^n}(a,d) \right]^{-1}.$$
By the universal property for $\wedge^2,$ this induces a pairing $\wedge^2 A[\ell^n]\times \wedge^2 A[\ell^n] \ra\mu_{\ell^n}^2$. One can check on the level of points that this pairing is perfect. 
Thus we can naturally identify $\wedge^2 A[\ell^n]$ with the Cartier dual of $\wedge^2A[\ell^n]\otimes\mu_{\ell^n}^{-1}.$  It follows that we may naturally regard the Weil pairing $w_{\ell^n}$ as an element of 
$\left( \wedge^2A[\ell^n]\otimes\mu_{\ell^n}^{-1} \right)(T).$ From now on we write $H(m)$ for $H\otimes\mu_{\ell^n}^{-m}$.

\begin{lem} \label{surjectivegeometricfibers}
Let $V/ S$ be a family of $g$-dimensional principally polarized abelian varieties with universal abelian scheme $\mathcal{A} \rightarrow V.$  The morphism
\begin{align*}
V_G &\xrightarrow{\pi} (\wedge^2 G)(1) \\
(A,\phi) &\mapsto \phi(w_{\ell^n}(A))
\end{align*}
is functorial and hence algebraic.  If $g \geq c(G),$ where the constant $c(G)$ depends only on $G,$ the morphism $\pi$ is surjective on geometric points.  
\end{lem}

\begin{proof}

Let $y \in (\wedge^2 G)(1)$ be a geometric point.  Let $A \in V(S)$ be an arbitrary abelian scheme.  Over the algebraically closed residue field $k(y),$ the group schemes $A[\ell^n]_{k(y)},\mu_{\ell^n}$ and 
$G_{k(y)}$ become constant, isomorphic to $\underline{\left(\ZZ / \ell^n \right)^{2g}},\underline{\ZZ / \ell^n }$ and $B = G(k(y))$ respectively.  

The Weil pairing 
$\omega\in\left(\wedge^2 A[\ell^n](1)\right)(k(y))$ is 
non-degenerate.  Surjections $A[\ell^n]_{k(y)} \rightarrow G_{k(y)}$ are equivalent to surjections of finite abelian groups $\left(\mathbb{Z} / \ell^n \right)^{2g} = A[\ell^n](k(y)) 
\twoheadrightarrow G(k(y)) = B.$  

Let $\omega_B \in \wedge^2 B$ correspond\footnote{This is a well-defined correspondence upon fixing a generator of $\mu_{\ell^n}$} to the geometric point $y \in \wedge^2 G(1).$  By \cite[Proposition 2.14]{Michael}, there is some constant $c(G)$ such that if $g \geq c(G),$ there exists some surjection $\phi: (\ZZ / \ell^n)^{2g} \twoheadrightarrow A$ for which $\phi(\omega) = \omega_A.$  The result follows.
\end{proof}

\subsection{Geometric monodromy and connected components}

\begin{prop} \label{connectedgeometricfibers}
Let $k$ be a field.  Let $\mathcal{A} \rightarrow V / k$ be a family of $g$-dimensional principally polarized abelian varieties with universal Weil pairing $\omega.$  Let $G / k$ be a finite \'{E}tale commutative group scheme.  Suppose that for every geometric point $\overline{z} \in V,$ the action of the geometric monodromy group $\pi_1(V,\overline{z}) = \mathrm{Gal}(k(\overline{\eta}) / k(\eta))$ on $\mathcal{A}[\ell^n](k(\overline{z})) \cong \left(\ZZ / \ell^n \right)^{2g}$ has image equal equal to the full symplectic group $\Sp \left( \mathcal{A}[\ell^n] (k(\overline{z})), \omega_{k(\overline{z})} \right).$  There is a constant $c(G)$ such that:  if $g \geq c(G),$ 
\begin{itemize}
\item
$\pi: V_G \rightarrow (\wedge^2 G)(1)$ is surjective on geometric points.

\item
For every geometric point $y \in (\wedge^2 G)(1),$ the fiber $\pi^{-1}(y)$ is connected.
\end{itemize}
\end{prop}

\begin{proof}
Let $\omega_0 \in \wedge^2 (\ZZ / \ell^n)^{2g}$ be non-degnerate. Let $A$ be a finite abelian $\ell$-group.  Consider the map
\begin{align*}
\Phi: \text{Surjections}( (\ZZ / \ell^n)^{2g}, A) &\rightarrow \wedge^2 A \\
T &\mapsto T(\omega_0)
\end{align*}
By \cite[Proposition 2.14]{Michael}, there is a constant $c(A)$ such that if $g \geq c(A), \Phi$ is surjective and forms a single orbit under the symplectic group $\Sp((\ZZ / \ell^n)^{2g}, \omega_0).$

Set $c(G) = c( B)$ where $B = G(\overline{k})$ for any algebraic closure $\overline{k} / k.$  By Lemma \ref{surjectivegeometricfibers}, the map $\pi$ is surjective on geometric points.  Let $\overline{y} \in \pi^{-1}(y)$ be a geometric point.  Let $\overline{z} = \pi_G(\overline{y}) \in V,$ where $\pi_G: V_G \rightarrow V$ is the forgetful map.   

Note that the fiber $\pi_G^{-1}(\overline{z})$ equals  
$$\mathrm{Surjections}(\mathcal{A}[\ell^n](k(\overline{z})), B) \cong \mathrm{Surjections}( (\ZZ / \ell^n)^{2g}, B);$$
the symbol $\cong$ means that there is an isomorphism $\mathcal{A}[\ell^n](k(\overline{z})) \rightarrow (\ZZ / \ell^n)^{2g}$ which is equivariant for the action of $\Sp( \mathcal{A}[\ell^n](k(\overline{z})), \omega_{k(\overline{z})} )$ on the left and of $\mathrm{Sp}((\mathbb{Z} / \ell^n)^{2g}, \omega_0)$ on the right.

The points of $\pi_G^{-1}(\overline{z})$ lying over $y,$ corresponding to $\omega_B \in \wedge^2 B,$ can be identified with 
$$\{T \in \mathrm{Surjections}((\ZZ / \ell^n)^{2g}, B): T(\omega_0) = \omega_B \}$$
By the above remarks, our assumption that the image of $\pi_1(V,\overline{z})$ in \newline$\Sp( \mathcal{A}[\ell^n](k(\overline{z})), \omega_{k(\overline{z})} )$ is surjective implies that $\pi_1(V,\overline{z})$ acts transitively on $\pi^{-1}(\overline{z}).$  It follows that $V_G$ is geometrically connected. 
\end{proof}

\begin{cor}
Let $k$ be a finite field of characteristic $p.$  Let $\ell' \neq \ell,p$ be a prime. Let $G / k$ be a finite \'{E}tale group scheme.  Same notation and hypotheses as in Proposition \ref{connectedgeometricfibers}.  Let $c(G)$ be the constant from Proposition \ref{connectedgeometricfibers} and assume that $g \geq c(G).$  Let $\overline{k} / k$ be an algebraic closure.  The map $\pi:V_G \rightarrow (\wedge^2 G)(1)$ induces a $\Gal(\overline{k}/k)$-equivariant isomorphism
$$H_{\mathrm{et}}^0((\wedge^2 G(1))_{\overline{k}}, \Q_{\ell'})  \xrightarrow{\pi^\ast} H_{\mathrm{et}}^0((V_G)_{\overline{k}}, \Q_{\ell'}).$$
In particular, 
$$\mathrm{tr} \left( \frob_{\overline{k} / k} | H_{\mathrm{et}}^0((V_G)_{\overline{k}}, \Q_{\ell'})  \right) = \# \left( (\wedge^2 G)(1) \right)(k).$$
\end{cor}

\begin{proof}
$\Gal(\overline{k} / k)$-equivariance follows becuase $V_G \rightarrow (\wedge^2 G)(1)$ is defined over $k.$  The map $\pi^\ast$ induces an isomorphism because $\pi^{-1}(y)$ is connected for every geometric point $y \in \wedge^2 G,$ by Proposition \ref{connectedgeometricfibers}.  
\end{proof}

\subsection{Comparison between moduli spaces of abelian varieties with level structure and EVW moduli spaces of covers}

Let $S$ be any base.  Let $\Conf_n $ be the moduli space of $n$ distinct unlabelled points in $\mathbb{A}^1.$ Let $\mathcal{C} \rightarrow \Conf_n$ be the associated family of hyperelliptic curves.  Let $\mathcal{A} \rightarrow \Conf_n$ denote the relative Jacobian of $\mathcal{C} / \Conf_n.$  There is an associated Torelli map $\mathcal{J}: \mathcal{C} / \Conf_n \rightarrow \mathcal{A} / \Conf_n.$  

Let $B$ be a finite abelian group of odd order.  Let $ \underline{B}_{\Conf_n}$ be the associated constant group scheme.  Let
$$\pi_{\underline{B}}: \Conf_{n,\underline{B}} \rightarrow \Conf_n$$
be the finite \'{E}tale cover described in \S \ref{modulispaces}.  Let $\phi: \mathcal{D}[\ell^n] / \Conf_n \twoheadrightarrow \underline{B} / \Conf_n$ be the associated universal quotient.  There is an associated finite \'{E}tale cover over $\Conf_{n,\underline{B}}$

$$\mathcal{D}[\ell^n] / \ker \phi  \rightarrow \mathcal{D} / \mathcal{D}[\ell^n] \cong \mathcal{D} = \mathcal{A}_{\Conf_{n,\underline{B}}},$$

where the isomorphism $\cong$ from the second map is the inverse of the projection isomorphism $\mathcal{D} \xrightarrow{\sim} \mathcal{D} / \mathcal{D}[\ell^n].$  Pulling back this finite 
\'{E}tale cover by the Torelli map $\mathcal{J}_{\Conf_{n,\underline{B}}}: \mathcal{C}_{\Conf_{n,\underline{B}}} \rightarrow \mathcal{A}_{\Conf_{n,\underline{B}}} $ defines a finite \'{E}tale 
cover $\mathcal{C}'\rightarrow \mathcal{C}_{\Conf_{n,\underline{B}}}$ of  $\Conf_{n,\underline{B}}$ with abelian Galois group $B.$  This finite \'{E}tale cover defines a morphism

$$\Phi: \Conf_{n,\underline{B}}  \rightarrow \mathrm{Hn}_{B \rtimes \langle \pm 1\rangle,n}^c / \Conf_n,$$

where $\mathrm{Hn}_{B \rtimes \langle \pm 1\rangle,n}^c$ denotes the moduli space of $B \rtimes \langle \pm 1 \rangle$-covers of $\mathbb{P}^1$ ramified at $\infty$ and having monodromy in the conjugacy class of involutions at $n$-finite punctures.  This is the moduli space considered by \cite{EVW}.  We refer the reader to \cite{EVW} and \cite{RW} for details on the algebraic construction of $\mathrm{Hn}_{B \rtimes \langle \pm 1\rangle,n}^c.$  In particular, we emphasize that both  $\Conf_{n,\underline{B}} / \Conf_n$ and  $\mathrm{Hn}_{B \rtimes \langle \pm 1\rangle,n}^c / \Conf_n$ are finite \'{E}tale.

\begin{prop}
Let $S = \overline{\mathbb{F}}_q$ be an algebraic closure of the finite field $\mathbb{F}_q.$  The morphism $\Phi: \Conf_{n,\underline{B}} / \Conf_n \rightarrow \mathrm{Hn}_{B \rtimes \langle \pm 1\rangle,n}^c / \Conf_n$ described above is an isomorphism.
\end{prop}

\begin{proof}
Because $\Conf_{n,\underline{B}} / \Conf_n$ and $\mathrm{Hn}_{B \rtimes \langle \pm 1\rangle,n}^c / \Conf_n$ are finite \'{E}tale, the morphism $\Phi$ is necessarily finite \'{E}tale.  By \cite[Proposition 8.7]{EVW}, $\Phi$ induces a bijection $\Conf_{n,\underline{B}}( \overline{\mathbb{F}}_q) \xrightarrow{\Phi} \mathrm{Hn}_{B \rtimes \langle \pm 1\rangle,n}^c( \overline{\mathbb{F}}_q).$  It follows that $\phi$ must have degree 1 and is thus an isomorphism.
\end{proof}

\begin{cor}\label{momentsare1general}
Assume now that $\ell$ is odd. Let $G$ be a finite \'{E}tale group scheme over $\F_q$ of order $\ell^n$,  and $\omega_G\in(\wedge^2 G)(1)(\F_q)$. 
For each $g$,  define $$\avg(G,\omega_G,g,q):=\frac{\#\{\phi\in\Sur(\Hom(\Pic^0(C)[\ell^n],G), \phi_*(\omega_{C,\ell^n})=\omega_G\}}{\#{\rm Conf}_g(\F_q)}$$
where $\omega_{C,\ell^n}$ is the weil-pairing.

Let $\delta^{\pm}(q,\omega_G)$ be the lower and upper limits of $\avg(G,\omega_G,g,q)$ as $g\ra\infty$. Then as $q\ra\infty$ and $n$ stays fixed, $\delta^{+}(q,\omega_G)$
 and $\delta^{-}(q,\omega_G)$ converge to 1.
\end{cor}

\begin{proof}

First, note that $\avg(G,\omega_G,g,q)\cdot|{\rm Conf}_g(\F_q)|$ is simply equal to the number of points on the subscheme $Y$ of $\textrm{Conf}_{g,G}$ which maps to
$\omega_G$ under the natural map to $(\wedge^2G)(1)$. By a result of Yu \cite{Yu}, the monodromy condition in Lemma \ref{connectedgeometricfibers} is satisfied, so 
$Y$ is geometrically connected. Moreover, by the discussion above the $\ell'$-adic cohomology of $\textrm{Conf}_{g,G}$ is the same as that of $\textrm{Hn}^c_{B,g}$ where 
$B=G(\ol{\F_q})\rtimes \Z/2\Z$ and $c$ is the conjugacy class of all involutions.  Thus, by \cite[Lemma 7.8]{EVW} we see that for all $i>0$, there is an integer $C(\ell^n)$ satisfying
$\dim H^i(\rm{Conf}_{g,G},\Q_{\ell'})\leq  C(\ell^n)^{i+1}$. The same bound therefore holds on the cohomology of $Y$. Thus, by the Lefschetz
trace formula we get that for $q>2C(\ell^n)^2$, $\#Y(\F_q) = q^n(1+O(C(G,\ell^n)/\sqrt{q}))$. The result follows since $|{\rm Conf}_g(\F_q)| = q^n-q^{n-1}$.

\end{proof}

\begin{remark} 

Note that if we stopped keeping track of $\omega_G$ and only cared about \'{e}tale group scheme, that the number of surjections to a group scheme $G$ approaches
$\#(\wedge^2G)(1)(\F_q)$.  If $G$ is a constant group scheme $\underline{B}$, this amounts to counting elements of $\wedge^2B$ which are killed by $q-1$. This is consistent
with the random model considered by Garton, and explains the failure of ordinary Cohen-Lenstra to hold if $q\not\equiv 1$ mod $\ell$.

\end{remark}

\section{Applications}
\subsection{Joint distribution}

Fix positive integers $n_1,\cdots,n_k.$  What is the joint distribution of the finite abelian $\ell$-groups $A(\F_{p^{n_1}})_\ell,\cdots,A(\F_{p^{n_k}})_\ell$ as $A$ varies through a family $V_g$ of $g$-dimensional principally polarized abelian varieties?

\subsubsection{\'{E}tale group schemes refine joint moments}
Let $n = \mathrm{lcm}(n_1,\cdots,n_k).$  Fix $G_1, G_2,\cdots,G_k$ finite abelian $\ell$-groups.  Let $M_n(A)$ denote the $\ZZ_\ell[x] / (x^n - 1)$-module $A(\F_{p^n})_\ell$ with is natural Frobenius action. Let $S_{G_1,\cdots,G_\ell}$ denote the set of isomorphism classes of $\ZZ_\ell[x] / (x^n - 1)$-modules $M$ for which 
$$M[x^{n_1} - 1] \cong G_1,\cdots, M [x^{n_k} - 1] \cong G_k.$$
Then

$$\mathrm{Prob}_{A \in V_g(\F_p)} \left( A(\F_{p^{n_1}})_\ell \cong G_1, \cdots , A(\F_{p^{n_k}})_\ell \cong G_k \right) = \sum_{M \in S_{G_1,\cdots,G_k}} \mathrm{Prob}_{A \in V_g(\F_p)} \left( M_n(A) \cong M \right).$$

So the distribution $A(\F_{p^n})_\ell$ as a $\ZZ_{\ell}[x]/(x^n - 1)$-module is a strictly more refined statistic than the joint distribution of $A(\F_{p^{n_1}}),\cdots,A(\F_{p^{n_k}}).$

\subsection{Results for the universal family of hyperelliptic curves} 
In this subsection, we spell out the consequences of our main theorems for the universal family of hyperelliptic curves in one special case.  For $x \in \mathrm{Conf}_g(\F_q),$ let $C_x$ denote the associated hyperelliptic curve and let $C_x^{\sigma}$ denote its quadratic twist.

For the ring $R = \ZZ_\ell[x] / (x^2 - 1)$ and finite $R$-module $M,$ let $M^{\pm}$ denote the $\pm 1$-eigenspaces of multiplication by $x.$

\begin{prop}\label{independence}
Suppose $\ell \nmid q^2 - 1$ and that $\ell \neq 2.$  Fix $\epsilon > 0.$  Fix a finite set $S$ of finite abelian $\ell$-groups.  Let $M_{A,B}$ denote the unique $R$-module whose $+1$-eigenspace equals $A$ and whose $-1$-eigenspace equals $B.$ There exists $Q(S)\gg 0$ such that if $q,g \geq Q(S)$ and $A,B \in S,$

$$\left| \mathrm{Prob}_{x \in \mathrm{Conf}_g(\F_q)} \left( \mathrm{Jac}(C_x)(\F_q)_\ell \cong A \text{ and } \mathrm{Jac}(C_x^{\sigma})(\F_q)_\ell \cong B \right) - \frac{c_R}{\# \mathrm{Aut}_R(M_{A,B})} \right| < \epsilon,$$

where $c_R$ is the normalizing constant from Theorem \ref{convmeas}.  

i.e. the class groups $\mathrm{Jac}(C_x)(\F_q)_\ell$ and $\mathrm{Jac}(C_x^{\sigma})(\F_q)_\ell$ behave almost independently for $g$ sufficiently large.
\end{prop}

\begin{proof}

Note that  $$\mathrm{Jac}(C_x)(\F_q)_\ell \cong A\textrm{ and } \mathrm{Jac}(C_x^{\sigma})(\F_q)_\ell \cong B  \iff \mathrm{Jac}(C_x)(\F_{q^2})_\ell\cong M_{A,B}.$$

By Proposition \ref{approximatemomentsimpliesapproximatelyCL} and Proposition \ref{momentsare1general}, for $q$ sufficiently large we have
$$\left|\mathrm{Prob}\left(\mathrm{Jac}(C_x)(\F_{q^2})_\ell\cong M_{A,B}\right) - \frac{c_R}{\# \Aut_R(M_{A,B})} \right|<\epsilon.$$
%
%

Because $R$ splits as a product, 
$$\frac{c_R}{\# \Aut_R(M_{A,B})} = \frac{c_{\ZZ_\ell}}{\# \Aut_{\ZZ_\ell}(A)} \cdot \frac{c_{\ZZ_\ell}}{\# \Aut_{\ZZ_\ell}(B)}.$$

The result follows.
\end{proof}

\end{document}